\documentclass[leqno]{article}
\usepackage[left=4cm, right=4cm]{geometry}
\usepackage{amssymb,amsmath,amsthm,enumerate}
\usepackage{secdot}
\usepackage{sectsty}
\sectionfont{\centering \large}
\sectionfont{\normalsize}
\theoremstyle{theorem}
\newtheorem{theorem}{Theorem}[section]

\newtheorem{lemma}{Lemma}[section]
\newtheorem{corollary}{Corollary}[section]
\theoremstyle{definition}
\newtheorem{definition}{Definition}[section]
\theoremstyle{definition}
\newtheorem{example}{Example}[section]
\theoremstyle{definition}

\numberwithin{equation}{section}
\begin{document}
	\title{On The Generalisation of Henstock-Kurzweil Fourier Transform}
	\author{S. Mahanta$^{1}$,\quad S. Ray$^{2}$\vspace{2mm}\\
		\em\small ${}^{1,2}$Department of Mathematics, Visva-Bharati, 731235, India.\\
		\em\small ${}^{1}$e-mail: sougatamahanta1@gmail.com\\
		\em\small ${}^{2}$e-mail: subhasis.ray@visva-bharati.ac.in}
	
	\date{}
	\maketitle
	
	\begin{abstract}
		In this paper, a generalised integral called the Laplace integral is defined on unbounded intervals, and some of its properties, including necessary and sufficient condition for differentiating under the integral sign, are discussed. It is also shown that this integral is more general than the Henstock-Kurzweil integral. Finally, the Fourier transform is defined using the Laplace integral, and its well-known properties are established.
	\end{abstract}
	\thanks{\bf Keywords:} Fourier transform, Henstock-Kurzweil integral, Denjoy integral, Laplace integral, Laplace derivative.\\\\
	\thanks{\bf Mathematics Subject Classification(2020):} 42A38, 26A39.

	\section{Introduction}
	If $f\colon\mathbb{R}\to\mathbb{R}$ is Lebesgue integrable, its Fourier transform is defined by $$\widehat{f}(y)=\int_{-\infty}^{\infty}f(x)e^{-2\pi iyx},$$ and its theory is also well established. Now the obvious question from the viewpoint of ``generalised integrals" is ``{\em can we replace the Lebesgue integral with a generalised integral in the definition of Fourier transform?"} It is Erik Talvila who first gave an affirmative answer to the above question in \cite{HKFT}. He used Henstock-Kurzweil integral to define the Fourier transform and proved its important properties. Furthermore, he pointed out that some beautiful results, e.g., Riemann-Lebesgue lemma, are not satisfied by the Henstock-Kurzweil Fourier transform. However, in \cite{mendoza2009some}, it is proved that the Riemann-Lebesgue lemma is satisfied in an appropriate subspace of the space of all Henstock-Kurzweil integrable functions on $\mathbb{R}$. Further results concerning Henstock-Kurzweil Fourier transform can be found in \cite{arredondo2020fourier, arredondo2018norm, mendoza2016note, morales2016extension, torres2015convolution, torres2010pointwise}.\vspace{2mm}
	
	\par In \cite{mahanta2021generalised}, a new generalised integral on bounded intervals called the Laplace integral is defined by the authors of this paper, which has continuous primitives and is more general than the Henstock-Kurzweil integral. In this paper, the concept of the Laplace integral is extended on $\mathbb{R}$ and then applied to define the Fourier transform. Moreover, the essential properties of the Fourier transform are studied in this general setting.\vspace{3mm}

	\noindent{\bf Notations:} Let $I$ be an interval bounded or unbounded. We use following notations throughout this paper.
	\begin{align*}
	&L^{1}(I)=\left\{ f\mid f\,\, \text{is Lebesgue integrable on $I$} \right\},\\
	&\mathcal{HK}(I)=\left\{ f\mid f\,\, \text{is Henstock-Kurzweil integrable on $I$} \right\},\\
	&\mathcal{BV}(I)=\left\{ f\mid f\,\, \text{is of bounded variation on $I$} \right\},\\
	&\mathcal{BV}(\pm\infty)=\left\{ f\mid f\,\, \text{is of bounded variation on $\mathbb{R}\setminus (-a,a)$ for some $a\in\mathbb{R}$}\right\},\\
	&V_{I}[f]= \text{Total variation of $f$ on $I$,}\\
	&V_{I}[f(\cdot, y)]= \text{Total variation of $f$ with respect to $x$ on $I$,}\\
	&\|f\|_{1}= \text{The $L^{1}$ norm of $f$}.
	\end{align*}

	\section{Preliminaries}
	\begin{definition}[\cite{mahanta2021generalised}]
		Let $f$ be Laplace integrable (see Definition 3.2 of \cite{mahanta2021generalised}) on a neighbourhood of $x$. If $\exists\, \delta>0$ such that the following limits
		\begin{equation*} 
		\lim\limits_{s\rightarrow \infty}s\int_{0}^{\delta}e^{-st}f(x+t)\,dt\quad\text{and}\quad\lim\limits_{s\rightarrow \infty}s\int_{0}^{\delta}e^{-st}f(x-t)\,dt
		\end{equation*}
		exist and are equal, then the common value is denoted by $LD_{0}f(x)$. And we say $f$ is Laplace continuous at $x$ if $LD_{0}f(x)=f(x)$.
	\end{definition}
	
	\begin{definition}
		Let $f$ be Laplace integrable on a neighbourhood of $x$. If $\exists\, \delta>0$ such that the following limits
		\begin{equation*}
		\lim\limits_{s\rightarrow \infty}s^{2}\int_{0}^{\delta}e^{-st}[f(x+t)-f(x)]\,dt\quad\text{and}\quad\lim\limits_{s\rightarrow \infty}(-s^{2})\int_{0}^{\delta}e^{-st}[f(x-t)-f(x)]\,dt
		\end{equation*}
		exist and are equal, then we say $f$ is Laplace differentiable at $x$ and the common value is denoted by $LD_{1}f(x)$.
	\end{definition} 
	\noindent If $f$ is a function of two variables, say $x$ and $y$, then we define the Laplace derivative of $f$ with respect to $x$ as the Laplace derivative of $f_{y}(x)=f(x,y)$ ($y$ is assumed to be constant) and we denote it by $LD_{1x}f$. Definition of $LD_{1y}f$ is similar. For further properties on Laplace derivative and Laplace continuity see \cite{OLD, TLD1, TLD2, OLC}.

	\section{The Laplace integral on unbounded intervals}
	In Definition 3.4 of \cite{borkowski2018applications}, Denjoy-Perron type integrals on unbounded intervals are defined and then it is proved that it is equivalent to Henstock-Kurzweil integral on unbounded intervals (Theorem 3.2 of \cite{borkowski2018applications}). Similarly, we shall define the Laplace integral on unbounded intervals and establish its properties. Due to similarity, most of the results will be given without proof.

	\begin{definition}\label{Laplace integral}
		Let $I=[a,\infty]$ and let $f:I\to\mathbb{R}$. Then we say $f$ is Laplace integrable on $[a,\infty)$ or on $[a,\infty]$ if
		\begin{enumerate}[\upshape (a)]
			\item $f$ is Laplace integrable on $[a,c]$ for $c\geqslant a$ and
			\item $\lim\limits_{c\rightarrow\infty}\int_{a}^{c}f$ exists.
		\end{enumerate} 
		In this case we write $\int_{a}^{\infty}f=\lim\limits_{c\rightarrow\infty}\int_{a}^{c}f$. The set of all Laplace integrable functions on $[a,\infty)$ will be denoted by $\mathcal{LP}[a,\infty)$ or by $\mathcal{LP}[a,\infty]$.
	\end{definition}
	
	\noindent Integrability on $(-\infty,b]$ or on $[-\infty,b]$ can be defined analogously. We shall say that $f$ is integrable on $\mathbb{R}$ or on $\overline{\mathbb{R}}$ (\,$=[-\infty,\infty]$\,) if there is some $a\in \mathbb{R}$ such that $f$ is integrable on both $(-\infty,a]$ and $[a,\infty)$, and we write $\int_{-\infty}^{\infty}f=\int_{-\infty}^{a}f+\int_{a}^{\infty}f$. From now on, we assume all integrals are Laplace integral unless otherwise stated.\vspace{2mm}
	\par From Section 3 of \cite{mahanta2021generalised} and Definition 3.4, Theorem 3.2 of \cite{borkowski2018applications}, it is evident that Denjoy-Perron integral or Henstock-Kurzweil integral on unbounded intervals is a particular case of Laplace integral; however, the following example will ensure that $\mathcal{LP}(\mathbb{R})\setminus\mathcal{HK}(\mathbb{R})$ is non-empty.

	\begin{example}\label{TheOnleExample}
		Let $f:[0,1]\to\mathbb{R}$ be the function defined in Equation (4.12) of \cite{mahanta2021generalised} and let $\phi=LD_{1}f$ on $[0,1]$ (see Theorem 4.2 of \cite{mahanta2021generalised}). Now define
		\[
		F(x)=
		\begin{cases}
		f(x)\quad& \text{if $x\in [0,1]$}\\
		\,\,\,\,0\quad& \text{if $x\in\mathbb{R}\setminus [0,1]$}.
		\end{cases}
		\]
		Then $F$ is continuous and Laplace differentiable on $\mathbb{R}$. If we denote $LD_{1}F$ by $\Phi$, then we get
		\[
		\Phi(x)=
		\begin{cases}
		\phi(x)\quad& \text{if $x\in [0,1]$}\\
		\,\,\,\,0\quad& \text{if $x\in\mathbb{R}\setminus [0,1]$}.
		\end{cases}
		\]
		By Theorem 4.3 of \cite{mahanta2021generalised} and Definition \ref{Laplace integral}, we get that $\Phi\in \mathcal{LP}(\mathbb{R})\setminus\mathcal{HK}(\mathbb{R})$.
	\end{example}
	
	\begin{theorem}[{\bf Cauchy criterion}]\label{Cauchy criterion}
		Let $I=[a,\infty]$ and let $f:I\to\mathbb{R}$ be such that $f\in\mathcal{LP}[a,c]$ for all $c\geqslant a$. Then $f\in\mathcal{LP}(I)$ if and only if for any $\epsilon>0$ there is a $K(\epsilon)\geqslant a$ such that $q>p\geqslant K(\epsilon)$ implies $|\int_{p}^{q}f|\leqslant \epsilon$.
	\end{theorem}

	\begin{theorem}
		Let $I=[a,\infty]$. Then 
		\begin{enumerate}[\upshape (a)]
			\item if $f,g \in \mathcal{LP}(I)$ and $c\in\mathbb{R}$, then $cf+g\in \mathcal{LP}(I)$ and $\int_{I}(cf + g)=c\int_{I}f + \int_{I}g$.
			\item if $f,g \in \mathcal{LP}(I)$ and $f\leqslant g$, then $\int_{I}f\leqslant\int_{I}g$.
			\item if $c\in I$ and $f \in \mathcal{LP}(I)$, then $\int_{a}^{\infty}f=\int_{a}^{c}f+\int_{c}^{\infty}f$.
		\end{enumerate}
	\end{theorem}
	
	\begin{theorem}[{\bf Fundamental theorem of calculus}]\label{Fundmental}
		Let $I=[a,\infty]$ and let $f:I\to\mathbb{R}$. If $f\in \mathcal{LP}(I)$ and $F(x)=\int_{a}^{x}f$, then $LD_{1}F=f$ a.e. on $I$.
	\end{theorem}
	\begin{proof}
		Let $I_{0}=[a,a+1)$ and $I_{n}=[a+n, a+n+1)$, $n\in \mathbb{N}$. Then $I=\bigcup_{n=0}^{\infty}I_{n}$. Now apply Theorem 5.5 of \cite{mahanta2021generalised} on each $I_{n}$.
	\end{proof}

	\begin{theorem}[{\bf Du Bois-Reymond's}]
		Let $I=[a,\infty]$, let $f:I\to\mathbb{R}$ and let $g:I\to\mathbb{R}$. If
		\begin{enumerate}[\upshape (a)]
			\item $f\in \mathcal{LP}[a,c]$ for all $c\geqslant a$ and $F(x)=\int_{a}^{x}f$ is bounded on $[a,\infty]$,
			\item $g\in L^{1}(I)\cap\mathcal{BV}(I)$ and $G(x)=\int_{a}^{x}g$ for all $x\in I$,
			\item $\lim\limits_{x\rightarrow \infty}F(x)G(x)$ exists,
		\end{enumerate}
		then $fG\in \mathcal{LP}(I)$.
	\end{theorem}
	\begin{proof}
		As $F$ is bounded, $Fg\in L^{1}(I)$ which implies $\lim_{x\rightarrow\infty}\int_{a}^{x}Fg$ exists and equal to $\int_{a}^{\infty}Fg$. Let $x\in[a,\infty)$, then by Theorem 6.1 of \cite{mahanta2021generalised}, we have
		\begin{equation}\label{Eq Du Bois-Reymond's}
		\int_{a}^{x}fG=F(x)G(x)-\int_{a}^{x}Fg.
		\end{equation}
		Now applying the last assumption on \eqref{Eq Du Bois-Reymond's}, we get $fG\in \mathcal{LP}(I)$.
	\end{proof}
	\begin{corollary}[\bf Integration by parts]\label{Int by parts}
		Let $I=[a,\infty]$, let $f\in \mathcal{LP}(I)$, let $g\in L^{1}(I)\cap\mathcal{BV}(I)$ and let $G(x)=\int_{a}^{x}g$, $x\in I$. Then $fG\in \mathcal{LP}(I)$ and
		\begin{equation*}
		\int_{a}^{\infty}fG=\lim\limits_{x\rightarrow\infty}[F(x)G(x)]-\int_{a}^{\infty}Fg,
		\end{equation*} 
		where $F(x)=\int_{a}^{x}f$.
	\end{corollary}
	
	\begin{lemma}\label{Holder}
		Let $[a,b]\subseteq \mathbb{R}$, let $f\in \mathcal{LP}(\mathbb{R})$ and let $g\in \mathcal{BV}(\mathbb{R})$. If $G(x)=\int_{a}^{x}g$, then
		\begin{equation*}
		\left|\int_{a}^{b}fG\right|\leqslant \left|\int_{a}^{b}f\right|\inf\limits_{x\in[a,b]} |G(x)| + \|f\|_{[a,b]}V_{[a,b]}[G]\leqslant \left[\inf\limits_{x\in[a,b]} |G(x)| + V_{[a,b]}[G]\right]\|f\|_{[a,b]}.
		\end{equation*}
	\end{lemma}
	\noindent Proof is similar to that of Lemma 24 of \cite{HKFT}.

	
	\begin{theorem}\label{sequence of integrals}
		Let $I=[a,\infty]$, let $f\in\mathcal{LP}(I)$, let $(g_{n})$ be a sequence in $L^{1}(I)\cap\mathcal{BV}(I)$ and let $G_{n}(x)=\int_{a}^{x}g_{n}$. If $(g_{n})$ is of uniform bounded variation, then
		\[
		\lim\limits_{n\rightarrow \infty}\int fG_{n}=\int f(\,\,\lim\limits_{n\rightarrow \infty} G_{n}).
		\]
	\end{theorem}
	\noindent For the proof see Theorem 7.2 of \cite{mahanta2021generalised}.

	\begin{lemma}\label{HK-LP}
		Let $I=[a,\infty]$, where $a\in\mathbb{R}$. Then $\mathcal{LP}(I)\cap\mathcal{BV}(I)=\mathcal{HK}(I)\cap\mathcal{BV}(I)$.
	\end{lemma}
	\begin{proof}
		It is enough to prove that $f\in \mathcal{LP}(I)\cap\mathcal{BV}(I)$ implies $f\in\mathcal{HK}(I)$. Note that $f\in \mathcal{LP}(I)\cap\mathcal{BV}(I)$ implies $f\in \mathcal{LP}([a,b])\cap\mathcal{BV}([a,b])$, where $a<b<\infty$. As bounded Laplace integrable functions on finite intervals are Lebesgue integrable (see Corollary 5.2 of \cite{mahanta2021generalised}), $f\in \mathcal{HK}([a,b])$. Now, as
		\begin{equation*}
		\lim\limits_{b\rightarrow\infty}\,\,\,(\mathcal{HK})\int_{a}^{b}f=\lim\limits_{b\rightarrow\infty}\,\int_{a}^{b}f=\int_{a}^{\infty}f,
		\end{equation*}
		Hake's theorem (Theorem 12.8 of \cite[p.~291]{AMTOIntegration}) implies $f\in\mathcal{HK}(I)$.
	\end{proof}
	
	
	\section{A necessary and sufficient condition of Laplace differentiation under the integral sign}
	One may notice that almost the whole paper \cite{HKFT} depends on Lemma 25 of \cite{HKFT}. Furthermore, Theorem 4 of \cite{talvila2001necessary} has a crucial role in proving that lemma. In this section, we also establish results similar to Theorem 4 of \cite{talvila2001necessary}. However, before that, we need to define the concept of a $\mathcal{LP}$-primitive.

	\begin{definition}[{\bf $\mathcal{LP}$-primitive}]
		Let $I=[a,b]\subseteq \overline{\mathbb{R}}$ and let $F\colon I\to \mathbb{R}$ be continuous. We say $F$ is a $\mathcal{LP}$-primitive on $I$ if $LD_{1}F$ exists a.e. on $I$, $LD_{1}F\in \mathcal{LP}(I)$ and $\int_{\alpha}^{\beta}LD_{1}F=F(\beta)-F(\alpha)$ for $\alpha, \beta\in I$.
	\end{definition}

	\begin{theorem}\label{Nec-Suff Prop}
		Let $I=[\alpha,\beta]\times[a,b]\subseteq \overline{\mathbb{R}}\times\overline{\mathbb{R}}$ and let $f\colon I\to \mathbb{R}$. Suppose $f(.\,,y)$ is a $\mathcal{LP}$-primitive on $[\alpha,\beta]$ for a.e. $y\in (a,b)$. Then $F(x)=\int_{a}^{b}f(x,y)\,dy$ is a $\mathcal{LP}$-primitive and $LD_{1}F(x)=\int_{a}^{b}LD_{1x}f(x,y)\,dy$ for almost every $x\in (\alpha, \beta)$ if and only if
		\begin{equation}\label{interchange-integral}
		\int_{s}^{t}\int_{a}^{b}LD_{1x}f(x,y)\,dydx=\int_{a}^{b}\int_{s}^{t}LD_{1x}f(x,y)\,dxdy \qquad \text{for all $s,t \in [\alpha, \beta]$.}
		\end{equation} 
	\end{theorem}
	\begin{proof}
		Let $F$ be a $\mathcal{LP}$-primitive and $LD_{1}F(x)=\int_{a}^{b}LD_{1x}f(x,y)\,dy$ for almost every $x\in (\alpha, \beta)$. Let $s,t\in [\alpha, \beta]$. Then
		\begin{align*}
		\int_{s}^{t}\int_{a}^{b}LD_{1x}f(x,y)\,dydx&= F(t)-F(s)\\
		&=\int_{a}^{b}[f(t,y)-f(s,y)]\,dy=\int_{a}^{b}\int_{s}^{t}LD_{1x}f(x,y)\,dxdy.
		\end{align*}
		Conversely, let us assume \eqref{interchange-integral} holds. Let $x_{0}\in (\alpha, \beta)$ be fixed. Then for $x\in (\alpha, \beta)$, we get
		\begin{align*}
		\int_{x_{0}}^{x}\int_{a}^{b}LD_{1t}f(t,y)\,dydt&=\int_{a}^{b}\int_{x_{0}}^{x}LD_{1t}f(t,y)\,dtdy\\
		&=\int_{a}^{b}[f(x,y)-f(x_{0},y)]\,dy=F(x)-F(x_{0}).
		\end{align*}
		Hence by Theorem \ref{Fundmental}, we get $F$ is a $\mathcal{LP}$-primitive and $LD_{1}F(x)=\int_{a}^{b}LD_{1x}f(x,y)\,dy$ for almost every $x\in (\alpha, \beta)$.
	\end{proof}

	\begin{corollary}\label{Nec-Suff Cor}
		Let $I=[\alpha,\beta]\times[a,b]\subseteq \overline{\mathbb{R}}\times\overline{\mathbb{R}}$ and let $g\colon I\to \mathbb{R}$. Suppose $g(.\,,y)\in \mathcal{LP}[\alpha,\beta]$ a.e. $y\in (a,b)$. Define $G(x)=\int_{a}^{b}\int_{\alpha}^{x}g(t,y)\,dtdy$. Then $G$ is a $\mathcal{LP}$-primitive and $LD_{1}G(x)=\int_{a}^{b}g(x,y)\,dy$ for a.e. $x\in (\alpha,\beta)$ if and only if $\int_{s}^{t}\int_{a}^{b}g(x,y)\,dydx=\int_{a}^{b}\int_{s}^{t}g(x,y)\,dxdy$ for $[s,t]\subseteq[\alpha,\beta]$.
	\end{corollary}


	\section{Convolution}
	Let $f\colon\mathbb{R}\to\mathbb{R}$ and $g\colon\mathbb{R}\to\mathbb{R}$. Then the convolution $f*g$ of $f$ and $g$ is defined by 
	\begin{equation}\label{convolution}
	f*g(x)=\int\limits_{\mathbb{R}}f(x-y)g(y)\,dy\qquad\text{for all $x\in\mathbb{R}$,}
	\end{equation} 
	provided the integral exists. We all know the basic properties of convolution while the integral is Lebesgue integral or Henstock-Kurzweil integral (see Section 3 of \cite{HKFT}). Here, we discuss similar properties when the integral is Laplace integral. However, before that, we need to prove the following Lemmas.

	\begin{lemma}\label{transformation}
		Let $f\in\mathcal{LP}[a,\infty]$. Then $\int_{a}^{\infty}f(x)\,dx=\int_{a-y}^{\infty}f(x+y)\,dx$.
	\end{lemma}
	\begin{proof}
		Let $\epsilon>0$ be given. Let $U$ and $V$ be respectively major and minor functions of $f$ on $[a,b]$ with $V(a)=U(a)=0$ and $0\leqslant U(b)-V(b)\leqslant \epsilon$, where $a<b<\infty$. Then it is quite obvious that $U\circ\phi$ and $V\circ\phi$ are respectively major and minor functions of $f\circ\phi$ on $[a-y,b-y]$ with $V\circ\phi(a-y)=U\circ\phi(a-y)=0$ and $0\leqslant U\circ\phi(b-y)-V\circ\phi(b-y)\leqslant \epsilon$, where $\phi(x)=x+y$ for all $x\in [a-y,b-y]$. Thus if $F(x)=\int_{a}^{x}f(t)\,dt$ for $x\in[a,b]$, then $F\circ\phi(x)=\int_{a-y}^{x}f\circ\phi(t)\,dt$ for $x\in[a-y,b-y]$ which implies
		\begin{equation*}
		\int_{a}^{b}f(t)\,dt=F(b)=F\circ\phi(b-y)=\int_{a-y}^{b-y}f(t+y)\,dt.
		\end{equation*}
		Now letting $b\rightarrow \infty$, we get the desired result.
	\end{proof}

	\begin{lemma}\label{most useful lemma}  
		Let $f\in \mathcal{LP}(\mathbb{R})$, let $G\colon\mathbb{R}^{2}\to\mathbb{R}$ and let $\mathfrak{I}$ be the collection of all open intervals of $\mathbb{R}$. Moreover, let $\partial_{x}G(x,y)=g(x,y)$. Now define the iterated integrals
		\begin{align*}
		&\mathrm{I}_{1}(A,B)=\int\limits_{y\in B}\int\limits_{x\in A}f(x)G(x,y)\,dxdy,\\
		&\mathrm{I}_{2}(A,B)=\int\limits_{x\in A}\int\limits_{y\in B}f(x)G(x,y)\,dydx,
		\end{align*}
		where $(A,B)\in\mathfrak{I}\times\mathfrak{I}$. If $A\in\mathfrak{I}$ is bounded and if 
		\begin{enumerate}[\upshape (a)] 
			\item $V_{O}[g(\cdot\,,y)]\in L^{1}(B)$, for any interval $O\subseteq A$,
			\item $\left|g(x,y)\right|\leqslant \eta_{J}(y)$ and $\left|G(x,y)\right|\leqslant \kappa_{J}(y)$ for a.e. $(x,y)\in J\times B$, where $J$ is any interval in $A$ and $\eta_{J},\,\kappa_{J}\in L^{1}(B)$,
		\end{enumerate}
		then $\mathrm{I}_{1}(A,B)$ exists, and $\mathrm{I}_{1}(A,B)=\mathrm{I}_{2}(A,B)$. In addition, if $\mathrm{I}_{2}(\mathbb{R},B)$ exists, then $\mathrm{I}_{1}(\mathbb{R},B)=\mathrm{I}_{2}(\mathbb{R},B)$.
	\end{lemma}
	
	\begin{proof}
		Let $J$ be any bounded interval and let $F(x)=\int_{-\infty}^{x}f$. Then by the second condition of this Lemma, we get $\int_{J}\int_{-\infty}^{\infty}\left|F(x)g(x,y)\right|\,dydx<\infty$, proving that $Fg\in L^{1}(J\times\mathbb{R})$ for any bounded interval $J$. So for $-\infty<a< t< b<\infty$ and $(\alpha,\beta)\subseteq\mathbb{R}$, we get
		\begin{equation}\label{Fubini's theorem}
		\int_{\alpha}^{\beta}\int_{a}^{t}F(x)g(x,y)\,dxdy=\int_{a}^{t}\int_{\alpha}^{\beta}F(x)g(x,y)\,dydx.
		\end{equation}
		Let $((a,b),(\alpha,\beta))\in \mathfrak{I}\times\mathfrak{I}$, where $(a,b)$ is bounded. For $t\in (a,b)$, define
		\[
		H_{a}(t)=\int_{\alpha}^{\beta}\int_{a}^{t}f(x)G(x,y)\,dxdy.
		\]
		Then applying integration by parts and \eqref{Fubini's theorem}, we get
		\begin{equation}\label{finiteness of H_{a}}
		H_{a}(t)=F(t)\int_{\alpha}^{\beta}G(t,y)\,dy-\int_{a}^{t}\int_{\alpha}^{\beta}F(x)g(x,y)\,dydx.
		\end{equation}
		Applying the second condition on \eqref{finiteness of H_{a}}, it can be proved that $H_{a}(t)$ is continuous on $(a,b)$. Let $\phi(x)=\int_{\alpha}^{\beta}G(x,y)\,dy$. Then on any bounded interval $J$, $\phi$ is absolutely continuous. Furthermore, as $\left|\partial_{x}G(x,y)\right|=\left|g(x,y)\right|\leqslant \eta_{J}(y)\in L^{1}((\alpha, \beta))$, we get $\phi^{'}(x)=\int_{\alpha}^{\beta}g(x,y)\,dy$. Thus by Corollary 6.1 of \cite{mahanta2021generalised}, we have $LD_{1}H_{a}(t)=f(t)\int_{\alpha}^{\beta}G(t,y)\,dy=f(t)\phi(t)$ for a.e. $t\in(a,b)$. Let $P:=\{a=x_{0},x_{1},...,x_{n}=b\}$ be any partition of $[a,b]$. Then by the first condition, we get
		\[ 
		\sum_{i=0}^{n-1}\left| \phi^{'}(x_{i+1})-\phi^{'}(x_{i})\right|\leqslant \int_{\alpha}^{\beta}\sum_{i=0}^{n-1}\left| g(x_{i+1},y)-g(x_{i},y)\right|\,dy\leqslant \int_{\alpha}^{\beta}V_{[a,b]}[g(\cdot\,,y)]\,dy<\infty,
		\]
		proving that $\phi^{'}\in \mathcal{BV}[a,b]$. Thus $f\phi\in \mathcal{LP}[a,b]$. Moreover, integrating $f\phi$, we can prove that $H_{a}$ is one of its primitive. Now Theorem \ref{Nec-Suff Prop} implies that $\mathrm{I}_{1}(J, (\alpha,\beta))=\mathrm{I}_{2}(J, (\alpha,\beta))$, where $(J, (\alpha,\beta))\in\mathfrak{I}\times\mathfrak{I}$ and $J$ is bounded.\vspace{2mm}
		
		\par As it is assumed that $I_{2}(\mathbb{R},B)$ exists, for $a\in \mathbb{R}$, we have
		\begin{align}\label{H(t)}
		\begin{split}
		\int\limits_{[a,\infty]}\int\limits_{B}f(x)G(x,y)\,dydx&=\lim\limits_{t\rightarrow \infty}\int\limits_{[a,t]}\int\limits_{B}f(x)G(x,y)\,dydx\\
		&=\lim\limits_{t\rightarrow \infty}\int\limits_{B}\int\limits_{[a,t]}f(x)G(x,y)\,dxdy.
		\end{split}
		\end{align}
		Thus $\lim\limits_{t\rightarrow \infty}H_{a}(t)$ exists. Define
		\begin{align}\label{H*(t)}
		H^{*}_{a}(t)=
		\begin{cases}
		H_{a}(t)&\text{if $a\leqslant t<\infty$,}\\
		\lim\limits_{x\rightarrow \infty}H_{a}(x)&\text{if $t=\infty$}.
		\end{cases}
		\end{align}
		Then $H^{*}_{a}$ is continuous on $[a,\infty]$ and
		\begin{align*}
		LD_{1}H^{*}_{a}(t)=LD_{1}H_{a}(t)=\int_{B}f(t)G(t,y)\,dy\qquad\text{for a.e. $t\in\mathbb{R}$.}
		\end{align*}
		Moreover, existence of $I_{2}(\mathbb{R},B)$ implies that $\int_{a}^{\infty}LD_{1}H^{*}_{a}(t)\,dt$ exists. Now as $H_{a}$ is a $\mathcal{LP}$-primitive on every bounded intervals in $\mathbb{R}$, applying \eqref{H(t)} and \eqref{H*(t)} we have 
		\[
		\int_{\alpha}^{\beta}LD_{1}H^{*}_{a}(t)\,dt= H^{*}_{a}(\beta) - H^{*}_{a}(\alpha)\qquad\text{for all $\alpha, \beta\in [a,\infty]$}
		\]
		which implies that $H^{*}_{a}$ is a $\mathcal{LP}$-primitive. Therefore, Corollary \ref{Nec-Suff Cor} implies that
		\[
		I_{1}((a,\infty),B)=I_{2}((a,\infty),B).
		\]
		Similarly, we can prove that $I_{1}((-\infty,a),B)=I_{2}((-\infty,a),B)$, and this completes the proof.
	\end{proof}

	\begin{theorem}
		Let $f\colon\mathbb{R}\to\mathbb{R}$ and $g\colon\mathbb{R}\to\mathbb{R}$. Then
		\begin{enumerate}[\upshape (a)] 
			\item  $f*g=g*f$, provided \eqref{convolution} exists.
			\item if $f\in\mathcal{LP}(\mathbb{R})$, $h\in L^{1}(\mathbb{R})$ and $g''\in L^{1}(\mathbb{R})$, then $(f*g)*h=f*(g*h)$.
			\item for $z\in\mathbb{R}$, $\tau_{z}(f*g)=(\tau_{z}f)*g=f*(\tau_{z}g)$, where $\tau_{z}f(x)=f(x-z)$.
			\item if $A=\{ x+y\mid x\in supp(f),\,\, y\in supp(g) \}$, then $supp(f*g)\subseteq\overline{A}$.    		
		\end{enumerate}
	\end{theorem} 
	\begin{proof}
		Proof of (a) is a straight forward consequence of Lemma \ref{transformation} and that of (c) and (d) are same as in the case of Lebesgue Fourier transformation. So we give the proof of (b) only.\vspace{2mm}
		\par Applying (a), we have
		\begin{equation*}
		(f*g)*h(x)=\int\int f(y)g(x-y-z)h(z)\,dydz.
		\end{equation*}
		Let $G^{x}(y,z)=g(x-y-z)h(z)$. As $g'$ is bounded and $g''$, $h$ are Lebesgue integrable on $\mathbb{R}$, $G^{x}(y,z)$ satisfies all conditions of Lemma \ref{most useful lemma}. Thus 
		\begin{align*}
		(f*g)*h(x)&=\int\int f(y)g(x-y-z)h(z)\,dzdy\\
		&=\int\int f(y)g*h(x-y)\,dy= f*(g*h)(x). \qedhere
		\end{align*}
	\end{proof}

	\begin{theorem}\label{Four-Conv-Lem}
		Let $f\in\mathcal{LP}(\mathbb{R})$, let $g\in L^{1}(\mathbb{R})\cap\mathcal{BV}(\mathbb{R})$, let $F(x)=\int_{-\infty}^{x}f$ and let $G(x)=\int_{-\infty}^{x}g$. If $F$ and $G$ are Lebesgue integrable and $G(\infty)=\lim_{x\rightarrow\infty}G(x)=0$, then 
		\[
		\int_{-\infty}^{t}\int_{-\infty}^{\infty}f(x-y)G(y)\,dydx=\int_{-\infty}^{\infty}\int_{-\infty}^{t}f(x-y)G(y)\,dxdy.
		\]
		Moreover, we have
		\begin{align}\label{conv-norm}
		\begin{split}
		&\|f*G\|_{1}\leqslant \|F\|_{1}\|g\|_{1},\\
		&\|f*G\|_{\mathcal{A}}\leqslant \|f\|_{\mathcal{A}}\|G\|_{1},
		\end{split}
		\end{align}
		where $\|f\|_{\mathcal{A}}=\sup_{x\in \mathbb{R}}\left|\int_{-\infty}^{x}f\right|$, the Alexiewicz's norm (see \cite{LFDIF}) of $f$.
	\end{theorem}
	\begin{proof}
		Let $h(x)=\int_{-\infty}^{\infty}f(x-y)g(y)\,dy$ for $x\in\mathbb{R}$. Then using integration by parts, we have 
		\begin{equation}\label{conv-eq1}
		h(x)=\int_{-\infty}^{\infty}F(x-y)g(y)\,dy=F*g(x).
		\end{equation}
		Now as both $F$ and $g$ are Lebesgue integrable, using Fubini's theorem and integration by parts, we have 
		\begin{align}\label{conv-eq2}
		\begin{split}
		\int_{-\infty}^{t}h(x)\,dx&=\int_{-\infty}^{\infty}\left(\int_{-\infty}^{t}F(x-y)\,dx\right)g(y)\,dy\\
		&=\int_{-\infty}^{\infty}F(t-y)G(y)\,dy=\int_{-\infty}^{\infty}\int_{-\infty}^{t}f(x-y)G(y)\,dxdy
		\end{split}
		\end{align}
		which implies $\int_{-\infty}^{t}\int_{-\infty}^{\infty}f(x-y)G(y)\,dydx=\int_{-\infty}^{\infty}\int_{-\infty}^{t}f(x-y)G(y)\,dxdy$. From \eqref{conv-eq1} and \eqref{conv-eq2}, we shall get \eqref{conv-norm}.
	\end{proof}

	\section{Fourier transform}    
	Let $f\colon\mathbb{R}\to \mathbb{R}$. Then the Fourier transform of $f$ is defined by 
	\begin{equation}\label{Fourier Transform}
	\widehat{f}(y)=\int_{-\infty}^{\infty}f(x)e^{-2\pi iyx}\,dx,
	\end{equation}
	provided the integral exists for $y\in \mathbb{R}$.\vspace{2mm} 
	\par Now, if we denote $G_{y}(x)=e^{-2\pi iyx}$, then it is easy to verify that $G'_{y}$ is of bounded variation on finite intervals. Hence if the integral in \eqref{Fourier Transform} exists, it implies that $f\in \mathcal{LP}_{loc}(\mathbb{R})$, where by $\mathcal{LP}_{loc}(\mathbb{R})$ we mean the set of all locally Laplace integrable functions. Now we establish some basic properties of the Fourier transform in this general setting.
	
	\begin{theorem}[\bf Existence theorems]\noindent
		\begin{enumerate}[\upshape (a)] 
			\item If $f\in\mathcal{LP}_{loc}(\mathbb{R})$ and $f$ is Lebesgue integrable in a neighbourhood of infinity then $\widehat{f}$ exists. 
			\item If $f\in \mathcal{LP}(\mathbb{R})\cap\mathcal{BV}(\pm\infty)$, then $\widehat{f}$ exists on $\mathbb{R}$.
		\end{enumerate}
	\end{theorem}

	\begin{proof} 
		\begin{enumerate}[\upshape (a)] 
			\item Let $a\in (0,\infty)$ be such that $f\in L^{1}(\mathbb{R}\setminus(-a,a))$. Then $\int_{|x|\geqslant a}f(x)e^{-2\pi iyx}\,dx$ exists. Now, as $g_{y}(x)=(-2\pi iy)e^{-2\pi iyx}$ is of bounded variation on $[-a,a]$,
			\[
			\int_{-a}^{a}f(x)e^{-2\pi iyx}\,dx
			\]
			exists. Thus $\widehat{f}$ exists on $\mathbb{R}$.
			\item It is a straightforward consequence of the previous part, Lemma \ref{HK-LP}, and Theorem 3.1 of \cite{mendoza2009some}.\qedhere
		\end{enumerate}
	\end{proof}
	
	\noindent The above theorem implies that $\widehat{\Phi}$ (see Example \ref{TheOnleExample}) exists in our setting; however, $\widehat{\Phi}$ does not exist in the sense of Henstock-Kurzweil Fourier transform since $\Phi$ is not locally Henstock integrable.

	\begin{theorem}
		Suppose $f,g\in \mathcal{LP}(\mathbb{R})$.
		\begin{enumerate}[\upshape (a)] 
			\item If $\widehat{f}$ exists, then $\widehat{\tau_{\zeta}f}(y)=e^{-2\pi i\zeta y}\widehat{f}(y)$ and $\tau_{\eta}\widehat{f}=\widehat{h}$, where $h(x)=e^{2\pi i\eta x}f(x)$.
			\item 	Let $\widehat{f}$ exists at $y\in\mathbb{R}$, let $g\in L^{1}(\mathbb{R})\cap\mathcal{BV}(\mathbb{R})$, let $G(x)=\int_{-\infty}^{x}|g|$ and let $F_{y}(x)=\int_{-\infty}^{x}e^{-2\pi iyt}f(t)\,dt$. If $F_{y}$ and $G$ are Lebesgue integrable and $G(\infty)=\lim_{x\rightarrow\infty}G(x)=0$, then 
			\[
			\widehat{f*G}(y)=\widehat{f}(y)\widehat{G}(y).
			\]
			\item Let $x^{n}f\in \mathcal{LP}(\mathbb{R})$ and let $\widehat{x^{n}f}$ exists for $n=0,1$. Then $\dfrac{d\widehat{f}}{dy}=\widehat{h}$ a.e. on $\mathbb{R}$, where $h(x)=(-2\pi ix)f(x)$.
			\item Let $f,f'\in \mathcal{LP}(\mathbb{R})$. If $\widehat{f}$ exists, then $\widehat{f'}$ exists and $\widehat{f'}(y)=(2\pi iy)\widehat{f}(y)$.
			\item Let $f\in \mathcal{LP}(\mathbb{R})$. If
			\begin{enumerate}[\upshape (i)]
				\item $f$ has compact support, then $\widehat{f}$ is continuous on $\mathbb{R}$.
				\item $f\in \mathcal{BV}(\pm\infty)$, then $\widehat{f}$ is continuous on $\mathbb{R}\setminus\{0\}$.
			\end{enumerate}
			\item {\upshape({\bf Riemann-Lebesgue Lemma})} If $f\in \mathcal{LP}(\mathbb{R})\cap\mathcal{BV}(\mathbb{R})$, then $\lim_{|t|\rightarrow\infty} \widehat{f}(t)=0$.
		\end{enumerate}
	\end{theorem}  
	\begin{proof} Proof of (a) and (b) follow from Lemma \ref{transformation} and Theorem \ref{Four-Conv-Lem}, respectively, and proof of (f) follows from Theorem 5.2 of \cite{mendoza2009some} and Lemma \ref{HK-LP}. So we shall prove the rest.
		\begin{enumerate}[\upshape (a)]
			\setcounter{enumi}{2} 
			\item Let $A=\mathbb{R}$ and $B=[n,n+1]$, $n\in\mathbb{Z}$. Let $h(x)=(-2\pi ix)f(x)$ and let $G(x,y)=e^{-2\pi ixy}$. Then
			\begin{enumerate}[\upshape (i)]
				\item $g(x,y)=\partial_{x}G(x,y)=(-2\pi iy)e^{-2\pi ixy}$,
				\item $V_{[a,b]}[g(\cdot\,,y)]=2\pi (b-a)y\in L^{1}(B)$, for any compact interval $[a,b]\subseteq A$,
				\item $\left|g(x,y)\right|=2\pi y \in L^{1}(B)$ and $\left|G(x,y)\right|=1 \in L^{1}(B)$ for $(x,y)\in A\times B$, and
				\item $\mathrm{I}_{2}(A,B)=\int_{x\in A}\int_{y\in B}h(x)G(x,y)\,dydx=[\widehat{f}(n+1)-\widehat{f}(n)]\in \mathbb{R}$
			\end{enumerate}
			which implies $\mathrm{I}_{1}(A,B)=\int_{y\in B}\int_{x\in A}h(x)G(x,y)\,dxdy=\mathrm{I}_{2}(A,B)$ (see Lemma \ref{most useful lemma}). Hence, we get
			\begin{equation}\label{derFourTr}
			\int_{s}^{t}\int_{-\infty}^{\infty}h(x)G(x,y)\,dxdy=\int_{-\infty}^{\infty}\int_{s}^{t}h(x)G(x,y)\,dxdy,
			\end{equation}
			where $s,t\in [n,n+1]$. Now, as $\widehat{f}(y)=\int_{-\infty}^{\infty}H(x,y)\,dx$, where $H(x,y)=f(x)G(x,y)$, and $\partial_{y} H(x,y)=h(x)G(x,y)$, by \eqref{derFourTr} and Theorem \ref{Nec-Suff Prop}, we have $\dfrac{d\widehat{f}}{dy}=\widehat{h}$ a.e. on $[n,n+1]$ for all $n\in \mathbb{Z}$ and hence a.e. on $\mathbb{R}$.		
			
			\item As $f, f'$ are integrable, $\lim_{|x|\rightarrow \infty}f(x)=0$. For $[u,v]\subseteq \mathbb{R}$, we have
			\begin{equation*}
			\int_{u}^{v}f'(x)e^{-2\pi ixy}\,dx= \left[e^{-2\pi ixy}f(x)\right]_{u}^{v}+(2\pi iy)\int_{u}^{v}f(x)e^{-2\pi ixy}\,dx.    	
			\end{equation*} 
			Now taking $u\rightarrow -\infty$ and $v\rightarrow \infty$, we arrive at our conclusion.
			\item 
			\begin{enumerate}[\upshape (i)] 
				\item Let $a\in(0,\infty)$ be such that $f(x)=0$ for $|x|>a$. Let $y_{0}\in \mathbb{R}$ be arbitrary and let $I_{y_{0}}=[y_{0}-1,y_{0}+1]$. If we denote, $g_{y}(x)=e^{-2\pi iyx}$, then it is easy to verify that the set $\{g'_{y}\mid y\in I_{y_{0}}\}$ is of uniform variation. Hence, by Theorem 7.2 of \cite{mahanta2021generalised}, we get
				\begin{align*}
				\lim\limits_{y\rightarrow y_{0}}\widehat{f}(y)&=\lim\limits_{y\rightarrow y_{0}}\int_{-\infty}^{\infty}f(x)e^{-2\pi iyx}\,dx\\
				&=\lim\limits_{y\rightarrow y_{0}}\int_{-a}^{a}f(x)e^{-2\pi iyx}\,dx=\widehat{f}(y_{0}).
				\end{align*}
				\item Let $a\in (0,\infty)$ be such that $f\in \mathcal{BV}(\mathbb{R}\setminus(-a,a))$. Note that $f=f_{1}+f_{2}+f_{3}$, where
				\[
				f_{1}=f\chi_{[-\infty,-a]},\qquad f_{2}=f\chi_{[-a,a]},\qquad f_{3}=f\chi_{[a,\infty]}.
				\]
				To prove $\widehat{f}$ is continuous on $\mathbb{R}\setminus\{0\}$ it is enough to show that $\widehat{f_{1}},\widehat{f_{3}}$ are continuous on $\mathbb{R}\setminus\{0\}$. Now Lemma \ref{HK-LP} implies $f_{1},f_{3}\in \mathcal{HK}(\mathbb{R})\cap\mathcal{BV}(\mathbb{R})$ and hence Theorem 4.2 of \cite{mendoza2009some} completes the proof.\qedhere
			\end{enumerate}
		\end{enumerate}
	\end{proof}

	\begin{lemma}\label{interchanging fourier cap}
		Let $\psi$ and $\phi$ be real-valued function on $\mathbb{R}$ and let $\widehat{\psi}$ exists a.e. on $\mathbb{R}$. If $\phi(y)$, $y\phi(y)$ and $y^{2}\phi(y)$ are Lebesgue integrable on $\mathbb{R}$ and if $\int_{-\infty}^{\infty}\psi\widehat{\phi}$ exists, then $\int_{-\infty}^{\infty}\psi\widehat{\phi}=\int_{-\infty}^{\infty}\widehat{\psi}\phi$.
	\end{lemma}
	\begin{proof}
		Let $G(x,y)=\phi(y)e^{-2\pi iyx}$ and $g(x,y)=\partial_{x}G(x,y)$. Then 
		\begin{enumerate}[\upshape (a)] 
			\item $V_{[a,b]}[g(\cdot\,,y)]=\int_{a}^{b}\left|\partial_{x}g(x,y)\right|\,dx=\int_{a}^{b}4\pi^{2}y^{2}\left|\phi(y)\right|\,dx=4\pi^{2}(b-a)y^{2}\left|\phi(y)\right|$;
			\item $\int_{-\infty}^{\infty}V_{[a,b]}[g(\cdot\,,y)]\,dy=4\pi^{2}(b-a)\int_{-\infty}^{\infty}y^{2}\left|\phi(y)\right|\,dy=O_{[a,b]}<\infty$;
			\item $\left|g(x,y)\right|=2\pi\left|y\phi(y)\right|=\eta(y)\in L^{1}(\mathbb{R})$;
			\item $\left|G(x,y)\right|=\left|\phi(y)\right|=\kappa(y)\in L^{1}(\mathbb{R})$. 
		\end{enumerate}
		Therefore, by Lemma \ref{most useful lemma}, we get $\int_{-\infty}^{\infty}\psi\widehat{\phi}=\int_{-\infty}^{\infty}\widehat{\psi}\phi$.
	\end{proof}
	\noindent Let us denote
	\begin{equation}
	\widetilde{f}(x)=\int_{-\infty}^{\infty}\widehat{f}(t)e^{2\pi ixt}\,dt
	\end{equation}

	\begin{theorem}[{\bf The Inversion Theorem}]
		Let $f\colon\mathbb{R}\to\mathbb{R}$ be such that $\widehat{f}$ exists almost everywhere on $\mathbb{R}$. Let $x_{0}\in\mathbb{R}$ be such that $\widetilde{f}$ exists at $x_{0}$. Then $f(x_{0})=\widetilde{f}(x_{0})$ provided    	
		\begin{equation}\label{inversion condition 1}
		\sup\limits_{x\in[-\delta,\delta]}\left|s\int_{-\delta}^{x}e^{-s|t|}\left(f(x_{0}+t)-f(x_{0})\right)\,dt\right|\rightarrow 0\qquad\text{as $s\rightarrow \infty$}
		\end{equation}
		for some $\delta(>0)$.
		
	\end{theorem}
	\begin{proof}
		For simplicity, we assume $x_{0}=0$. As $e^{-t^{2}}$, $te^{-t^{2}}$ and $t^{2}e^{-t^{2}}$ belong to $L^{1}(\mathbb{R})$, by Lemma \ref{interchanging fourier cap}, we get
		\[
		\int_{-\infty}^{\infty}e^{-\lambda^{2}\pi t^{2}}\widehat{f}(t)\,dt=\int_{-\infty}^{\infty}\widehat{e^{-\lambda^{2}\pi t^{2}}}f(t)\,dt.
		\]
		Now as $G_{\lambda}(t)=\frac{d}{dt}(e^{-\lambda^{2}\pi t^{2}})\in L^{1}(\mathbb{R})\cap\mathcal{BV}(\mathbb{R})$ and the set $\{V_{\mathbb{R}}[G_{\lambda}]\mid 0\leqslant \lambda \leqslant 1\}$ is uniformly bounded on $\mathbb{R}$,
		Theorem \ref{sequence of integrals} implies
		\begin{equation}\label{Inversion eqn 1}
		\lim\limits_{\lambda\rightarrow 0^{+}}\int_{-\infty}^{\infty}e^{-\lambda^{2}\pi t^{2}}\widehat{f}(t)\,dt=\int_{-\infty}^{\infty}\widehat{f}(t)\,dt=\widetilde{f}(0).
		\end{equation}
		Let $\delta>0$ be such that \eqref{inversion condition 1} holds. 
		Then
		\begin{align*}
		\int_{-\infty}^{\infty}\widehat{e^{-\lambda^{2}\pi t^{2}}}f(t)\,dt&=\int_{-\infty}^{\infty}\lambda^{-1}e^{-\pi t^{2}/\lambda^{2}}f(t)\,dt\\
		&=\int\limits_{\left|t\right|<\delta}\lambda^{-1}e^{-\pi t^{2}/\lambda^{2}}f(t)\,dt + \int\limits_{\left|t\right|\geqslant\delta}\lambda^{-1}e^{-\pi t^{2}/\lambda^{2}}f(t)\,dt.
		\end{align*}
		Let $I_{1}=\int_{\left|t\right|<\delta}\lambda^{-1}e^{-\pi t^{2}/\lambda^{2}}f(t)\,dt$, $I_{2}=\int_{\left|t\right|\geqslant\delta}\lambda^{-1}e^{-\pi t^{2}/\lambda^{2}}f(t)\,dt$ and $J_{\delta}=\mathbb{R}\setminus(-\delta, \delta)$. Then by simple calculation, we can prove that
		\[
		H_{\lambda}(t)=\frac{d}{dt}(\lambda^{-1}e^{-\pi t^{2}/\lambda^{2}})\in L^{1}(\mathbb{R})\cap\mathcal{BV}(\mathbb{R})
		\]
		and the set $\{V_{\mathbb{R}}[H_{\lambda}]\mid 0\leqslant \lambda \leqslant \epsilon\}$ is uniformly bounded on $J_{\delta}$, where $\epsilon\in [0,1)$ is sufficiently small. Thus again,
		Theorem \ref{sequence of integrals} implies that $I_{2}\rightarrow 0$ as $\lambda\rightarrow 0^{+}$. Now 
		\begin{align*}
		I_{1}&=f(0)\int_{-\delta}^{\delta}\lambda^{-1}e^{-\pi t^{2}/\lambda^{2}}\,dt + \int_{-\delta}^{\delta}\lambda^{-1}e^{-t/\lambda}(f(t)-f(0))e^{-\pi t^{2}/\lambda^{2}\, +\, t/\lambda}\,dt\\
		&=f(0)\int_{-\delta}^{\delta}\lambda^{-1}e^{-\pi t^{2}/\lambda^{2}}\,dt + \int_{-\delta}^{\delta}g_{\lambda}(t)h_{\lambda}(t)\,dt,
		\end{align*}
		where $g_{\lambda}(t)=\lambda^{-1}e^{-t/\lambda}(f(t)-f(0))$ and $h_{\lambda}(t)=e^{-\pi t^{2}/\lambda^{2}\, +\, t/\lambda}$. By Lemma \ref{Holder}, we get 
		\begin{equation*}
		\left| \int_{-\delta}^{\delta}g_{\lambda}(t)h_{\lambda}(t)\,dt \right|\leqslant\left[ \inf\limits_{[-\delta,\delta]}\left|h_{\lambda}(t)\right| + V_{[-\delta,\delta]}[h_{\lambda}] \right]\|g_{\lambda}\|_{[-\delta,\delta]}.
		\end{equation*}
		For sufficiently small $\lambda$, it can be proved that $V_{[-\delta,\delta]}[h_{\lambda}]$ is independent of $\lambda$. Thus by \eqref{inversion condition 1}, we get
		\begin{align}\label{Inversion eqn 2}
		\begin{split}
		\lim\limits_{\lambda\rightarrow 0^{+}}\int_{-\infty}^{\infty}&\widehat{e^{-\lambda^{2}\pi t^{2}}}f(t)\,dt=\lim\limits_{\lambda\rightarrow 0^{+}} I_{1}=\lim\limits_{\lambda\rightarrow 0^{+}}f(0)\int_{-\delta}^{\delta}\lambda^{-1}e^{-\pi t^{2}/\lambda^{2}}\,dt\\
		&=f(0)\pi^{-1/2}\lim\limits_{\lambda\rightarrow 0^{+}}\int_{0}^{\frac{\pi\delta^{2}}{\lambda^{2}}}t^{-1/2}e^{-t}\,dt=f(0)\pi^{-1/2}\Gamma(1/2)=f(0).
		\end{split}
		\end{align}
		Equating \eqref{Inversion eqn 1} and \eqref{Inversion eqn 2}, we get $f(0)=\int_{-\infty}^{\infty}\widehat{f}(t)\,dt$.
	\end{proof}
	
	\begin{corollary}
		Let $f\colon\mathbb{R}\to\mathbb{R}$ be such that $\widehat{f}=0$ a.e. on $\mathbb{R}$. Then $f=0$ a.e. on $\mathbb{R}$.
	\end{corollary}
	
	\begin{proof}
		Let $I_{n}=[n,n+1]$ for $n\in \mathbb{Z}$. Then $f\in \mathcal{LP}(I_{n})$ for all $n\in \mathbb{Z}$. Now, Corollary 6.2 of \cite{mahanta2021generalised} implies that $f$ is Laplace continuous a.e. on $I_{n}$ and hence \eqref{inversion condition 1} is satisfied a.e. on $I_{n}$ for all $n\in \mathbb{Z}$. Since $\widehat{f}=0$ a.e. on $\mathbb{R}$, we obtain $f=0$ a.e. on $I_{n}$ for all $n\in \mathbb{Z}$ which completes the proof.
	\end{proof}

	\section{Conclusions}
	The definition of Laplace integral depends on a generalised derivative called the Laplace derivative. Suppose it is possible to define the total Laplace derivative on $\mathbb{R}^{n}$ ($n\geqslant 2$) and establish its interrelations with the partial Laplace derivatives. In that case, it may be possible to find excellent applications of the Fourier transform (in our setting) to the generalised PDEs, i.e., to the PDEs using partial Laplace derivatives, which, we hope, will be an interesting problem to deal with.\vspace{3mm}

    \noindent {\bf Acknowledgements}. The UGC fellowship of India supports this research work of the first author (Serial no.- 2061641179, Ref. no.- 19/06/2016(i) EU-V and Roll no.- 424175 under the UGC scheme).

    \bibliographystyle{plain} 
    \bibliography{MyBibliography112.bib}
	
\end{document}